\documentclass{amsart}

\newtheorem{theorem}{Theorem}[section]

\newtheorem{proposition}[theorem]{Proposition}

\newtheorem{corollary}[theorem]{Corollary}

\theoremstyle{definition}

\usepackage{amsmath,amssymb,amsfonts}

\theoremstyle{remark}

\usepackage{amsfonts}
\usepackage{amssymb,t1enc,amsmath,psfrag}
\usepackage[dvips]{graphicx}

\numberwithin{equation}{section}
\renewcommand{\dim}{\mathrm{dim}}

\newcommand{\dist}{\mathrm{dist}}

\newcommand{\conv}{\mathrm{conv}}

\newcommand{\R}{\mathbb{R}}
\newcommand{\N}{\mathbb{N}}

\newcommand{\X}{\mathrm{X}}
\renewcommand{\H}{\mathrm{H}}
\newcommand{\Y}{\mathrm{Y}}
\newcommand{\Z}{\mathrm{Z}}
\newcommand{\B}{\mathbf{B}}
\renewcommand{\S}{\mathbf{S}}

\renewcommand{\S}{\mathbf{S}}

\renewcommand{\ker}{\mathrm{Ker}}
\renewcommand{\mod}{/}

\begin{document}
\title{Special symmetries of Banach spaces isomorphic to Hilbert spaces}

\author[J. Talponen]{Jarno Talponen}
\address{University of Helsinki, Department of Mathematics and Statistics, Box 68, FI-00014 University of Helsinki, Finland}
\email{talponen@cc.helsinki.fi}

\date{\today}
\begin{abstract}
In this paper Hilbert spaces are characterized among Banach spaces in terms of transitivity 
with respect to nicely behaved subgroups of the isometry group.
For example, the following result is typical: If $\X$ is a real Banach space isomorphic to a Hilbert space 
and convex-transitive with respect to the isometric finite-dimensional perturbations of the identity, 
then $\X$ is already isometric to a Hilbert space.
\end{abstract}
\maketitle

\section{Introduction}
The expression 'special symmetries' in the title refers to suitable subgroups of 
$\mathcal{G}(\X)=\{T\colon \X\rightarrow \X|\ T\ \mathrm{isometric\ automorphism}\}$ where $\X$ is a real Banach space.
We denote the closed unit ball of $\X$ by $\B_{\X}$ and the unit sphere by $\S_{\X}$.
The orbit of $x\in\S_{\X}$ with respect to a family $\mathcal{F}\subset L(\X)$ is given by
$\mathcal{F}(x)=\{T(x)|\ T\in \mathcal{F}\}$.
An inner product $(\cdot|\cdot)\colon \X\times \X\rightarrow \R$ is 
said to be \emph{invariant} with respect to $\mathcal{F}$ if $(T(x)|T(y))=(x|y)$ for each $x,y\in \X,\ T\in \mathcal{F}$.
The concept of an invariant inner product is an important tool applied frequently in this article.
We say that $\X$ is \emph{transitive}, \emph{almost transitive} or \emph{convex-transitive} with respect to 
$\mathcal{F}$ if $\mathcal{F}(x)=\S_{\X}$, $\overline{\mathcal{F}(x)}=\S_{\X}$ or
$\overline{\conv}(\mathcal{F}(x))=\B_{\X}$, respectively, for all $x\in \S_{\X}$. 
If $\mathcal{F}=\mathcal{G}(\X)$ above, then we will omit mentioning it. This article can be regarded as a part
of the field generated around the well-known open \emph{Banach-Mazur rotation problem}, which asks whether each transitive separable
Banach space is isometrically a Hilbert space. See \cite{BR} for an exposition of the topic.

In \cite{Ca} F. Cabello S\'{a}nchez studied the subgroup 
\[\mathcal{G}_{F}=\{T\in \mathcal{G}(\X)|\ \mathrm{Rank}(T-\mathrm{Id})<\infty\}\] 
consisting of the finite-dimensional perturbations of the identity. There a classical result appearing in \cite{Auerbach, Mazur}
is applied, namely, that each finite-dimensional Banach space admits an invariant inner product.
This motivated the work in \cite{Ca}, where an elegant proof was presented for the following result: 
\begin{theorem}\label{thmCa} 
If the norm of $\X$ is transitive with respect to $\mathcal{G}_{F}$, then $\X$ is isometric to 
a Hilbert space.
\end{theorem}
Cabello raised the question whether this result can be extended to the almost transitive
setting. It turns out here that the answer is affirmative under the additional assumption that $\X$ is isomorphic to a Hilbert space:
\begin{theorem}\label{thm1}
Let $\X$ be a Banach space isomorphic to a Hilbert space.
Then $\X$ is convex-transitive with respect to $\mathcal{G}_{F}$ if and only if $\X$ is isometric to
a Hilbert space. 
\end{theorem}
This paper is also motivated by the following problems posed in \cite{Ca0,Ca}:
\begin{enumerate}
\item[$\bullet$]{Is an almost transitive Banach space isometric to a Hilbert space if it is isomorphic to one?}
\item[$\bullet$]{Find ideals $J\subset L(\X)$ (with $F\subset J$) for which Theorem \ref{thmCa} remains true if condition 
$T-\mathrm{Id}\in F$ is replaced by $T-\mathrm{Id}\in J$ (here $F$ is the ideal of finite-rank operators).}
\end{enumerate}
Questions of this type are treated here, and we will also show that the existence of an invariant inner product on $\X$ 
is determined by the existence of invariant inner products separately with respect to finitely generated subgroups 
of $\mathcal{G}(\X)$ (see Theorem \ref{thm: ips}).

\subsection{Preliminaries}
We refer to \cite{BR}, \cite{GP}, \cite{HHZ} and \cite{S} for some background information.
Recall that a norm $||\cdot||$ on $\X$ is \emph{maximal} if $\mathcal{G}_{(\X,||\cdot||)}\subset \mathcal{G}_{(\X,|||\cdot|||)}$ 
for an equivalent norm $|||\cdot|||$ implies that $\mathcal{G}_{(\X,||\cdot||)}=\mathcal{G}_{(\X,|||\cdot|||)}$. 
If $\X$ is convex-transitive, then the norm of $\X$ is maximal, see \cite{Co}. 
We denote by $\mathrm{Aut}(\X)$ the group of isomorphisms $T\colon \X\rightarrow \X$.

Given a topological group $G$ we denote by $\mathrm{UCB}(G)$ the space of uniformly continuous bounded functions on $G$.
Here we consider the uniform structure $\Phi_{G}$ of $G$ as being generated by a basis of entourages of diagonal having the form
\begin{equation}\label{eq: uniform}
W=\{(g,h)\in G\times G|\ gh^{-1},\ g^{-1}h\in V\},
\end{equation}
where $V$ runs over a neighbourhood basis of $e$ in $G$.
The space $\mathrm{UCB}(G)$ is endowed with the $||\cdot||_{\infty}$-norm. 

For the sake of convenience we will enumerate the following condition: 
Suppose that there is a positive functional $F\in \mathrm{UCB}(G)^{\ast},\ ||F||=1,$ such that 
\begin{equation}\label{eq: amenable}
F(f(\cdot g))=F(f(\cdot)) \quad \mathrm{for\ all}\ f\in \mathrm{UCB}(G),\ g\in G.
\end{equation}
This type of condition can be viewed as a weaker version of amenability of $G$ (see \cite{R}). 
We note that the rotation group of $L^{p}$ with the strong operator topology is extremely amenable for $1\leq p<\infty$, see \cite{GP}.

Recall that the product topology of $\X^{\X}$ inherited by $L(\X)$ is called the strong operator topology (SOT).

We often consider subgroups $\mathcal{G}\subset \mathcal{G}(\X)$, which enjoy the following property:
\begin{enumerate}
\item[$(\ast)$]{Given $n\in \N$, $T_{1},\ldots,T_{n}\in \mathcal{G}$ and a finite-codimensional subspace $\Z\subset\X$ 
there exists a finite-codimensional subspace $\Y\subset\Z$ such that\\ 
$T_{1}(\Y)=\dots =T_{n}(\Y)=\Y$.}
\end{enumerate} 
Clearly $\mathcal{G}_{F}$ is an example of a subgroup of $\mathcal{G}(\X)$ satisfying $(\ast)$.

It is easy to see that if $\H$ is a Hilbert space, then $\mathcal{G}_{F}\subset \mathcal{G}(\H)$ is dense in $\mathcal{G}(\H)$
in the topology of uniform convergence on compact sets. On the other hand, given a Banach space $\X$ the group 
$\mathcal{G}(\X)$ is $\mathrm{SOT}$-closed in $\mathrm{Aut}(\X)$.

\section{Results}

\begin{theorem}
Let $\X$ be a maximally normed Banach space, which is isomorphic to a Hilbert space. 
Suppose that $\mathcal{G}(\X)$ endowed with the strong operator topology is amenable in the sense of condition \eqref{eq: amenable}.
Then $\X$ is isometrically isomorphic to a Hilbert space.
\end{theorem}

\begin{proof}
We may assume without loss of generality that
$(\X,||\cdot||)$ and $(\X,|\cdot|)$ are isomorphic via the identical mapping, where 
$|\cdot|$ is a norm induced by an inner product $(\cdot|\cdot)$ on $\X$. 
We denote by $\mathcal{G}(\X)=\mathcal{G}_{(\X,||\cdot||)}$ and $\mathcal{G}_{(\X,|\cdot|)}$ the 
corresponding rotation groups, and these are regarded with the strong operator topology.
Recall that $\Phi_{\mathcal{G}(\X)}$ is the natural uniformity given by the group $(\mathcal{G}(\X),\mathrm{SOT})$ applied to 
\eqref{eq: uniform}.

Observe that $T\mapsto (Tx|Ty)$ defines a $\Phi_{\mathcal{G}(\X)}$-uniformly continuous map $\mathcal{G}(\X)\rightarrow \R$
for each $x,y\in \X$. Indeed, this map is obtained by composing the $\Phi_{\mathcal{G}(\X)}$-$||\cdot||_{\X\oplus_{2}\X}$ 
uniformly continuous map $\mathcal{G}(\X)\rightarrow \X\oplus_{2} \X,\ T\mapsto (Tx,Ty)$ and the map 
$(Tx,Ty)\mapsto (Tx|Ty)$, which is $||\cdot||_{\X\oplus_{2}\X}$-uniformly continuous as $||\cdot||\sim |\cdot|$. 
To check that $T\mapsto (Tx,Ty)$ is uniformly continuous, first consider a standard entourage 
\[E=\{(x_{1},y_{1},x_{2},y_{2})\in \X\oplus_{2}\X \times \X\oplus_{2} \X:\ ||(x_{1},y_{1})-(x_{2},y_{2})||_{\X\oplus_{2}\X}<\epsilon\}\]
for some $\epsilon >0$. The preimage of this is 
\begin{equation*}
\begin{array}{rll}
& &\{(R,S)\in \mathcal{G}(\X)\times \mathcal{G}(\X):\ ||(Rx,Ry)-(Sx,Sy)||_{\X\oplus_{2}\X}<\epsilon\},\\
&\supset&\{(R,S)\in \mathcal{G}(\X)\times \mathcal{G}(\X):\ ||Tx-Sx||,||Ty-Sy||<\frac{\epsilon}{2}\}\\
&=&\{(R,S)\in \mathcal{G}(\X)\times \mathcal{G}(\X):\ ||x-T^{-1}Sx||,||y-T^{-1}Sy||<\frac{\epsilon}{2}\}.
\end{array}
\end{equation*}
Hence it suffices to pick $V=\{R\in \mathcal{G}(\X):\ ||x-Rx||,||y-Ry||<\frac{\epsilon}{2}\}$ in \eqref{eq: uniform}
to find an entourage of $\Phi_{\mathcal{G}(\X)}$ in the preimage of $E$. We obtain that $T\mapsto (Tx,Ty)$ is 
$\Phi_{\mathcal{G}(\X)}$-uniformly continuous.

According to the assumptions there is $F\in \mathrm{UCB}(\mathcal{G}(\X))^{\ast},\ ||F||=1,$ such that 
$F(f(\cdot g))=F(f(\cdot))$ for $f\in \mathrm{UCB}(\mathcal{G}(\X))$ and $g\in \mathcal{G}(\X)$.
For each $x,y\in \X$ we put 
\[[x|y]=F(\{(g(x)|g(y))\}_{g\in \mathcal{G}(\X)}).\]
This definition is sensible, since $g\mapsto (g(x)|g(y))$ defines an element 
in $\mathrm{UCB}(\mathcal{G}(\X))$ for each $x,y\in \X$.
We claim that $[\cdot|\cdot]$ defines an inner product on $\X$ such that 
$|||x|||\stackrel{\cdot}{=}\sqrt{[x|x]}$ is equivalent to $||\cdot||$.
Indeed, first note that $[\cdot|\cdot]\colon (\X,||\cdot||)\oplus_{2} (\X,||\cdot||)\rightarrow \R$ is defined and bounded, 
since $(\cdot|\cdot)\colon (\X,||\cdot||)\oplus_{2} (\X,||\cdot||)\rightarrow \R$ 
is bounded and $||F||=1$. By using the bilinearity of $(\cdot|\cdot)$ and the linearity of $F$ we obtain that 
$[\cdot|\cdot]$ is bilinear. Let $C\geq 1$ such that $C^{-2}||\cdot||^{2}\leq |\cdot|^{2}\leq C^{2}||\cdot||^{2}$.
Since $F$ is positive and norm-$1$, we get that 
\[C^{-2}||x||^{2}=\inf_{g}C^{-2}||g(x)||^{2}\leq F(\{(g(x)|g(x))\}_{g\in \mathcal{G}(\X)})\leq 
\sup_{g}C^{2}||g(x)||=C^{2}||x||,\]
where $x\in \X$ and the supremum and infimum are taken over $\mathcal{G}(\X)$. 
This means that $[\cdot|\cdot]$ is an inner product on $\X$ such that 
$|||\cdot|||$ is equivalent to $||\cdot||$.

Observe that 
\[[h(x)|h(y)]=F(\{(gh(x)|gh(y))\}_{g\in\mathcal{G}(\X)})=F(\{(g(x)|g(y))\}_{g\in\mathcal{G}(\X)})=[x|y]\]
for each $h\in \mathcal{G}(\X)$. The maximality of the norm of $(\X,||\cdot||)$ yields that 
$\mathcal{G}_{(\X,||\cdot||)}=\mathcal{G}_{(\X,|||\cdot|||)}$. The proof is completed by a standard argument
using the fact that $(\X,|||\cdot|||)$ is transitive.
\end{proof}

Suppose that $\X$ is a Banach space with two equivalent norms $||\cdot||$ and $|||\cdot|||$ such that the group
$\mathcal{G}$ generated by $\mathcal{G}_{(\X,||\cdot||)}\cup \mathcal{G}_{(\X,|||\cdot|||)}$ is 
operator norm bounded. Then there is one more equivalent norm $||||\cdot||||$ on $\X$ given by 
$||||x||||=\sup_{g\in\mathcal{G}}||g(x)||$ and this is $\mathcal{G}$-invariant. Consequently,
if the norms $||\cdot||$ and $|||\cdot|||$ are additionally maximal (resp. convex-transitive), then
$\mathcal{G}_{(\X,||\cdot||)}=\mathcal{G}_{(\X,|||\cdot|||)}$ (resp. $||\cdot||=c |||\cdot|||$ for some constant $c>0$).

The argument employed in the proof of \cite[Lemma 2]{Ca} can be modified to obtain the following dichotomy regarding 
the existence of invariant inner products.

\begin{theorem}\label{thm: ips}
Let $\X$ be a Banach space and $C\geq 1$. Suppose that for each $n\in\N$ and $T_{1},\ldots,T_{n}\in\mathcal{G}(\X)$
there exists an inner product $(\cdot|\cdot)_{\ast}\colon \X\times\X\rightarrow \R$ invariant under the rotations
$T_{1},\ldots,T_{n}$ such that $C^{-2}\ ||x||^{2}\leq (x|x)_{\ast}\leq C^{2}\ ||x||^{2}$ for each $x\in \X$.
Then there is already an inner product $(\cdot|\cdot)_{\X}\colon \X\times \X\rightarrow \R$, which is invariant under $\mathcal{G}(\X)$
and satisfies $C^{-2}\ ||x||^{2}\leq (x|x)\leq C^{2}\ ||x||^{2}$ for $x\in\X$.
\end{theorem}
\begin{proof}
We may assume without loss of generality that $\mathcal{G}(\X)$ is not finitely generated. 
Let $\mathcal{N}$ be the net of finitely generated subgroups of $\mathcal{G}(\X)$ ordered by inclusion.
By the assumptions we may assign for each $\gamma\in \mathcal{N}$ an inner product 
$(\cdot|\cdot)_{\gamma}\colon \X\times \X\rightarrow \R$ invariant under $\gamma$ and satisfying
$C^{-1}||x||^{2}\leq (x|x)_{\gamma}\leq C||x||^{2}$ for $x\in\X$. Observe that the sets 
$\{\gamma\in \mathcal{N}|\ \delta\subset \gamma\}$, where $\delta\in\mathcal{N}$, form a filter base of a filter 
$\mathcal{F}$ on $\mathcal{N}$. Let us extend $\mathcal{F}$ to an ultrafilter $\mathcal{U}$ on $\mathcal{N}$. 
Note that $\mathcal{U}$ is non-principal, since for each $\eta\in \mathcal{N}$ there is $\delta\in \mathcal{N}$
with $\eta\subsetneq \delta$, so that $\eta\notin \{\gamma\in \mathcal{N}|\ \delta\subset \gamma\}\in \mathcal{U}$.

Define $B\colon \X\times \X\rightarrow \R^{\mathcal{N}}$ by setting $B(x,y)=\{(x|y)_{\gamma}\}_{\gamma\in \mathcal{N}}$
for $x,y\in \X$. We will consider $\R^{\mathcal{N}}$ equipped with the usual point-wise linear structure. Then
$B$ becomes a symmetric and bilinear map. Moreover, $B(x,x)\geq 0$ point-wise for $x\in \X$.
Put $\overset{\rightarrow}{B}\colon \X\times\X\rightarrow \R,\ \overset{\rightarrow}{B}(x,y)=\lim_{\mathcal{U}}B(x,y)$ for $x,y\in\X$.
Indeed, the above limit exists and is finite for all $x,y\in \X$, 
since $(x|y)_{\gamma}\leq \sqrt{(x|x)_{\gamma}(y|y)_{\gamma}}\leq C^{2}\ ||x||\ ||y||$ for all $\gamma\in \mathcal{N},\ x,y\in\X$. 
Moreover, similarly we get that
$C^{-2}||x||^{2}\leq \overset{\rightarrow}{B}(x,x)\leq C^{2}||x||^{2}$ for all $x\in\X$. 
It follows that $\overset{\rightarrow}{B}$ is an inner product on $\X$. 

Observe that for all $T\in \mathcal{G}(\X)$ and $x,y\in\X$ we have that 
\[\{\gamma\in \mathcal{N}|\ (Tx|Ty)_{\gamma}=(x|y)_{\gamma}\}\supset \{\gamma\in\mathcal{N}|\ T\in \gamma\}\in \mathcal{F}\subset \mathcal{U}.\]
Hence $\overset{\rightarrow}{B}(Tx,Ty)=\overset{\rightarrow}{B}(x,y)$ for $T\in \mathcal{G}(\X)$ and $x,y\in\X$.
Consequently, $\overset{\rightarrow}{B}$ is the required inner product.
\end{proof}

It is not known if an almost transitive Banach space isomorphic to a Hilbert space is in fact isometric to a
Hilbert space (see \cite{Ca0}). The following consequence of Theorem \ref{thm: ips} provides a partial answer to this problem.

\begin{corollary}
Let $\X$ be a maximally normed Banach space, $\H$ a Hilbert space and $C\geq 1$. Suppose that for any $n\in\N$ and 
$T_{1},\ldots T_{n}\in\mathcal{G}(\X)$ there exists an isomorphism $\phi\colon \X\rightarrow \H$ such that 
$\max(||\phi||,||\phi^{-1}||)\leq C$ and $||\phi(x)||=||\phi(T_{i}x)||$ for $x\in \X$ and $i\in \{1,\ldots,n\}$.
Then $\X$ is already isometric to $\H$.
\end{corollary}
\begin{proof}
By putting $(x|y)_{\ast}=(\phi(x)|\phi(y))_{\H}$ for each $T_{1},\ldots ,T_{n}$ we obtain the assumptions of Theorem \ref{thm: ips}.
Let $(\cdot|\cdot)_{\X}\colon \X\times \X\rightarrow \R$ be the resulting inner product. Then $\X$ endowed with the 
norm $|||x|||\stackrel{\cdot}{=}\sqrt{(x|x)_{\X}}$ is transitive being a Hilbert space. Since $\X$ is maximally normed, we get that 
$\mathcal{G}_{(\X,||\cdot||)}=\mathcal{G}_{(\X,|||\cdot|||)}$. Thus $\X$ is transitive. It follows that 
$||\cdot||=c|||\cdot|||$ for some $c>0$, and hence $\X$ is a Hilbert space.
\end{proof}

\begin{theorem}\label{thm: innerprod}
Let $(\X,||\cdot||)$ be a Banach space, $(\H,(\cdot|\cdot)_{\H})$ an inner product space, $\mathcal{G}\subset \mathcal{G}(\X)$ 
a subgroup satisfying $(\ast)$ and let $S\colon \X\rightarrow \H$ be an isomorphism. 
Then there exists an inner product $(\cdot|\cdot)_{\X}$ on $\X$ such that 
\begin{enumerate}
\item[(1)]{$||S^{-1}||^{-2}\ ||x||^{2}\leq (x|x)_{\X}\leq ||S||^{2}\ ||x||^{2}$ for $x\in \X$.}
\item[(2)]{$(Tx|Ty)_{\X}=(x|y)_{\X}$ for $x,y\in \X$ and $T\in \overline{\mathcal{G}}^{\mathrm{SOT}}\subset L(\X)$.}
\end{enumerate}
\end{theorem}

\begin{proof}
It suffices to find $(\cdot|\cdot)_{\X}$, which satisfies conclusion (1) and conclusion (2) for merely $T\in \mathcal{G}$.
Indeed, given $T\in \overline{\mathcal{G}}^{\mathrm{SOT}}$ and $x,y\in \X$ there is a sequence 
$(T_{n})\subset \mathcal{G}$ such that $T_{n}(x)\rightarrow T(x)$ and $T_{n}(y)\rightarrow T(y)$ as $n\rightarrow\infty$.
This yields that $(T(x)|T(y))_{\X}-(x|y)_{\X}=\lim_{n\rightarrow \infty}((T_{n}(x)|T_{n}(y))_{\X}-(x|y)_{\X})=0$
by using the $\mathcal{G}$-invariance and the $||\cdot||$-continuity of $(\cdot|\cdot)_{\X}$.

Let $\mathcal{M}$ be the set of all pairs $(E,G)$, where $E\subset\X$ is a finite-codimensional subspace
and $G\subset\mathcal{G}$ is a finitely generated subgroup such that $T(E)=E$ for $T\in G$. 

According to the definition of $\mathcal{G}$ we obtain that $\bigcup_{(E,G)\in \mathcal{M}}G=\mathcal{G}$ and\\ 
$\bigcap_{(E,G)\in \mathcal{M}}E=\{0\}$. We equip $\mathcal{M}$ with the partial order $\leq$ defined as follows: 
$(E_{1},G_{1})\leq (E_{2},G_{2})$ if $E_{1}\supset E_{2}$ and $G_{1}\subset G_{2}$. So, $(\mathcal{M},\leq)$ is a directed set.

Suppose that $\Y\subset \H$ is a subspace of a Hilbert space and $\H\mod \Y$ is the corresponding quotient space. 
Then there exists a natural inner product on $\H\mod \Y$, namely 
\[(\widehat{x}^{\Y}|\widehat{y}^{\Y})_{\H\mod\Y}=(x-P_{\Y}x|y-P_{\Y}y)_{\H},\quad x,y\in \H,\]
where $\widehat{x}^{Y}=x+\Y,\ \widehat{y}^{\Y}=y+\Y$ and $P_{\Y}\colon \X\rightarrow \Y$ is the orthogonal projection onto $\Y$.

Given $(E,G)\in \mathcal{M}$ it holds that $T(E)=E$ for $T\in G$ and hence
the mapping $\widehat{T}_{E}\colon \X\mod E\rightarrow \X\mod E$ given by $\widehat{T}_{E}(\widehat{x}^{E})=T(x+E)$
defines a rotation on $\X\mod E$ for $T\in G$. Indeed, $||\widehat{x}^{E}||_{\X\mod E}=\dist(x,E)$ and 
$\dist(T(x),E)=\dist(x,E)$, as $T(E)=E$. Now, since $\X\mod E$ is finite-dimensional, the rotation group
$\mathcal{G}_{\X\mod E}$ is compact in the operator norm topology.

For each $(E,G)\in \mathcal{M}$ we define a map $\widehat{S}_{E}\colon \X\mod E\rightarrow \H\mod S(E)$
by $\widehat{S}_{E}(\widehat{x}^{E})=S(x+E)$. It is easy to see that 
\begin{equation}\label{eq: ipestimate}
\begin{array}{rl}
 &||S^{-1}||^{-2}\ ||\widehat{x}^{E}||_{\X\mod E}^{2}\leq (\widehat{S}_{E}(\widehat{x}^{E})|\widehat{S}_{E}(\widehat{x}^{E}))_{\H\mod S(E)}\\  & (\widehat{S}_{E}(\widehat{x}^{E})|\widehat{S}_{E}(\widehat{y}^{E}))_{\H\mod S(E)}\leq ||S||^{2}\ ||\widehat{x}^{E}||_{\X\mod E}\ ||\widehat{y}^{E}||_{\X\mod E}
\end{array}
\end{equation}
for $x,y\in \X$. Consider $\R^{\mathcal{M}}$ with the point-wise linear structure. 
Define a map $B\colon \X\times \X\rightarrow \R^{\mathcal{M}}$ by 
\[B(x,y)(E,G)=\int_{\mathcal{G}_{\X\mod E}}(\widehat{S}_{E}(\tau \widehat{x}^{E})|\widehat{S}_{E}(\tau \widehat{y}^{E}))_{\H\mod S(E)}\ \mathrm{d}\tau.\]
Above $\int_{\mathcal{G}_{\X\mod E}}$ is the invariant Haar integral over the compact group $\mathcal{G}_{\X\mod E}$.
The invariance of the integral yields that $B(Tx,Ty)(E,G)=B(x,y)(E,G)$ for $x,y\in \X,\ (E,G)\in\mathcal{M}$ and $T\in G$.
By using \eqref{eq: ipestimate} and the basic properties of the integral we obtain that
\begin{equation}\label{eq: bestimate}
\begin{array}{rl}
 &||S^{-1}||^{-2}\ ||\widehat{x}^{E}||_{\X\mod E}^{2}\leq B(x,x)(E,G)\\ 
 & B(x,y)(E,G)\leq ||S||^{2}\ ||\widehat{x}^{E}||_{\X\mod E}\ ||\widehat{y}^{E}||_{\X\mod E}
\end{array}
\end{equation}
for $x,y\in \X$ and $(E,G)\in \mathcal{M}$.

The family $\{\{\gamma\in\mathcal{M}|\ \gamma\geq \eta\}\}_{\eta\in\mathcal{M}}$ is a filter base on $\mathcal{M}$.
Let $\mathcal{U}$ be a non-principal ultrafilter extending $\{\{\gamma\in\mathcal{M}|\ \gamma\geq \eta\}\}_{\eta\in\mathcal{M}}$.
Put $(x|y)_{\X}=\lim_{\mathcal{U}}B(x,y)$ for $x,y\in\X$. It is easy to see that $(\cdot|\cdot)_{\X}$ is a bilinear mapping. 

According to \eqref{eq: bestimate} we get that $(x|y)_{\X}\leq ||S||^{2}||x||_{\X}||y||_{\X}$. 
Next, we aim to verify that $||S^{-1}||^{-2}||x||_{\X}^{2}\leq (x|x)_{\X}$. Towards this, we will check that 
$\sup_{(E,G)\in \mathcal{M}}||\widehat{x}^{E}||_{\X\mod E}$ $=||x||_{\X}$. Fix $x\in \S_{\X}$.
Assume to the contrary that $\sup_{(E,G)\in \mathcal{M}}||\widehat{x}^{E}||_{\X\mod E}=c<1$. 
Note that $\X$ is reflexive being isomorphic to $\H$. Thus the ball $x+c\B_{\X}$ is weakly compact.
Putting 
\[\{\{y\in E:\ ||x-y||\leq C\}\}_{(E,G)\in \mathcal{M}}\]
defines a net of non-empty closed convex subsets of $x+c\B_{\X}$. This net has a cluster point $z\in x+c\B_{\X}$
according to the weak compactness of $x+c\B_{\X}$. This means that $z\in \bigcap_{(E,G)\in\mathcal{M}} E$, which provides a 
contradiction, since $z\neq 0$. Consequently, \eqref{eq: bestimate} yields that
\[||S^{-1}||^{-2}\ ||x||_{\X}^{2}=||S^{-1}||^{-2}\ \lim_{\mathcal{U}}||\widehat{x}^{E}||_{\X\mod E}^{2}
\leq \lim_{\mathcal{U}}B(x,x)=(x|x)_{\X}.\]

Finally, we claim that $(Tx|Ty)_{\X}=(x|y)_{\X}$ for $x,y\in \X$ and $T\in \mathcal{G}$. 
Indeed, pick $T\in \mathcal{G}$ and $x,y\in\X$.
Then 
\begin{equation*}
\begin{array}{rll}
& &\{(E,G)\in \mathcal{M}:\ B(T(x),T(y))(E,G)=B(x,y)(E,G)\}\\
 &\supset&\{(E,G)\in \mathcal{M}:\ T\in G\}\in \mathcal{U},
\end{array}
\end{equation*}
so that $\lim_{\mathcal{U}}(B(Tx,Ty)-B(x,y))=0$.
\end{proof}

\begin{corollary}
Let $\X$ be a maximally normed space $\X$ isomorphic to a Hilbert space. 
Suppose that there is a subgroup $\mathcal{G}\subset \mathcal{G}(\X)$, which satisfies $(\ast)$ and 
$\mathcal{G}(\X)\subset \overline{\mathcal{G}}^{\mathrm{SOT}}$. Then $\X$ is isometrically a Hilbert space.

\qed
\end{corollary}

In Theorem \ref{thm: innerprod} the isomorphism $S$ was exploited in order to give bounds for the resulting
inner product $(\cdot|\cdot)_{\X}$. In \cite{Ca} a different approach was taken instead; namely the analogous construction
was suitably normalized by using a special point $x_{0}$. By suitably combining the arguments in \cite{Ca} and in the proof of 
Theorem \ref{thm: innerprod} we obtain the following result.
\begin{theorem}
Let $\X$ be a Banach space transitive with respect to a subgroup $\mathcal{G}\subset\mathcal{G}(\X)$, which satisfies $(\ast)$.
Then $\X$ is isometric to a Hilbert space.

\qed
\end{theorem}

Theorem \ref{thm1} is an immediate consequence of the following result. 
This result yields that $\X$ must be in particular almost transitive, and we note that there exists an alternative route to this fact, 
since spaces both convex-transitive and superreflexive are additionally almost transitive, see e.g. \cite{Fi}.

\begin{theorem}
Let $\X$ be a Banach space isomorphic to a Hilbert space and suppose $\mathcal{G}\subset \mathcal{G}(\X)$
is a subgroup, which satisfies $(\ast)$ and $\mathcal{G}_{F}\subset \mathcal{G}$. 
Then $\X$ is convex-transitive with respect to $\overline{\mathcal{G}}^{\mathrm{SOT}}\subset L(\X)$ 
if and only if $\X$ is isometric to a Hilbert space.
\end{theorem}
\begin{proof}
First note that a Hilbert space is transitive, in particular convex-transitive, and that 
$\mathcal{G}_{F}\subset\mathcal{G}(\H)$ is $\mathrm{SOT}$-dense in $\mathcal{G}(\H)$, so that the 'if' direction is clear.

Since $\X$ is isomorphic to a Hilbert space, we may apply Theorem \ref{thm: innerprod}
to obtain an $\overline{\mathcal{G}}^{\mathrm{SOT}}$-invariant inner product $(\cdot|\cdot)_{\X}$ on $\X$
such that $|||x|||^{2}=(x|x)_{\X}$ defines a norm equivalent with $||\cdot||_{\X}$. 
Clearly $|||\cdot|||$ is $\overline{\mathcal{G}}^{\mathrm{SOT}}$-invariant as well.  
By rescaling $|||\cdot|||$ we may assume without loss of generality that $||\cdot||_{\X}\leq |||\cdot|||$ 
and $\sup_{y\in \S_{(\X,|||\cdot|||)}}||y||_{\X}=1$. Put $C=\{x\in \X:\ |||x|||\leq 1\}$.

Fix $x\in \S_{(X,||\cdot||_{\X})}$ and $\epsilon>0$. Let $y\in \S_{(\X,|||\cdot|||)}$ be such that $||y||_{\X}>1-\frac{\epsilon}{2}$.
Since $(X,||\cdot||_{\X})$ is convex-transitive with respect to $\overline{\mathcal{G}}^{\mathrm{SOT}}$, we get
that $(1-\frac{\epsilon}{2})x\in \overline{\conv}^{||\cdot||_{\X}}(\{T(y)|T\in \overline{\mathcal{G}}^{\mathrm{SOT}}\})$. 
Since the norms $|||\cdot|||$ and $||\cdot||_{\X}$ are equivalent
we obtain that there is a convex combination 
$\sum a_{n}T_{n}(y)\in \conv(\{T(y)|T\in \mathcal{G}_{F}\})$ such that 
$|||(1-\frac{\epsilon}{2})x-\sum a_{n}T_{n}(y)|||<\frac{\epsilon}{2}$.
By noting that $|||\sum a_{n} T_{n}(y)|||\leq \sum a_{n} |||T_{n}(y)|||$ we get that 
$\sup_{T\in \overline{\mathcal{G}}^{\mathrm{SOT}}}|||T(y)|||\geq |||x|||-\epsilon$. 
Hence $|||y|||\geq |||x|||-\epsilon$ by using the $\overline{\mathcal{G}}^{\mathrm{SOT}}$-invariance of $|||\cdot|||$.
Since $\epsilon$ was arbitrary and $|||x|||\geq 1$, we deduce that $|||x|||=1$, and it follows that $||\cdot||_{X}=|||\cdot|||$. 
\end{proof}

Finally, we will take a different approach and characterize the Hilbert spaces in terms of the subgroup of rotations, 
that, instead of fixing a finite-codimensional subspace, rather fix a given $1$-dimensional subspace.     

\begin{proposition}
Let $\X$ be an almost transitive Banach space. Suppose that there exists $z_{0}\in \S_{\X}$ satisfying that for any 
$\epsilon>0$ and $x,y\in \S_{\X}$ with $\dist(x,[z_{0}])=\dist(y,[z_{0}])=1$, there is $T\in \mathcal{G}(\X)$
such that $||T(z_{0})-z_{0}||<\epsilon$ and $||T(x)-y||<\epsilon$. Then $\X$ is isometric to an inner product space.
\end{proposition}
\begin{proof}
It is well-known (see e.g. \cite{BR}) that almost transitive finite-dimensional spaces are isometric to Hilbert spaces. 
Hence we may concentrate on the case $\dim(\X)\geq 3$. Let $A,B\subset \X$ be $2$-dimensional subspaces such that $z_{0}\in A$.
Recall the classical result that a Banach space is isometric to a Hilbert space if and only if 
any couple of $2$-dimensional subspaces are mutually isometric (see \cite{AuMU}). 
Thus, in order to establish the claim, it suffices to verify that the subspaces $A$ and $B$
are isometric. 

Fix $0<\epsilon<1$, $x\in\S_{\X}\cap A$ such that $\dist(x,[z_{0}])=1$ and $w\in \S_{\X}\cap B$.
Let $f\in \S_{\X^{\ast}}$ be such that $f(w)=1$. 

Since $\X$ is almost transitive, there is $T_{1}\in \mathcal{G}(\X)$ such that 
$||T_{1}(w)-z_{0}||<\frac{\epsilon}{4}$. 
Define a linear operator $S\colon \X\rightarrow \X$ by $S(v)=T_{1}(v)+f(v)(z_{0}-T_{1}(w))$ for $v\in \X$ and note that
$S(w)=z_{0}$. Observe that $S$ is an isomorphism, since $||T_{1}-S_{1}||<\frac{\epsilon}{4}$.
Pick $y\in \S_{\X}\cap S(B)$ such that $\dist(y,[z_{0}])=1$. According to the assumptions there is $T_{2}\in\mathcal{G}(\X)$
such that $\max(||T_{2}(z_{0})-z_{0}||,||T_{2}(y)-x||)<\frac{\epsilon}{4}$. Let $g,h\in 2\B_{\X^{\ast}}$ be such that
$g(z_{0})=h(y)=1$, $y\in \ker(g)$ and $z_{0}\in \ker(h)$. 
Define a linear operator $U\colon \X\rightarrow \X$ by 
\begin{equation*}
\begin{array}{cc}
&U(v)=T_{2}(v)+g(v)(z_{0}-T_{2}(z_{0}))+h(v)(x-T_{2}(y))\quad \mathrm{for}\ v\in \X.
\end{array}
\end{equation*}
Note that $U(z_{0})=z_{0}$ and $U(y)=x$. Moreover, $||T_{2}-U||<\epsilon$, so that 
$U$ is an isomorphism. Observe that $U\circ S$ maps $B$ linearly onto $A$. 
We conclude that $A$ and $B$ are almost isometric, since $\epsilon$ was arbitrary.
Hence, being finite-dimensional spaces, $A$ and $B$ are isometric.
\end{proof}

\end{document}